\documentclass{amsart}
\def \rn {\mathbb{R}^n}
\def \sn {\mathbb{S}^n}
\def \rr {\mathbb{R}}
\def \nn {\mathbb{N}}
\def \crit {2^\star}
\newtheorem{thm}{Theorem}
\newtheorem{prop}{Proposition}

\title[Admissible $Q-$curvatures]{Admissible $Q-$curvatures under isometries for the conformal GJMS operators}
\author{Fr\'ed\'eric Robert}
\address{Institut Elie Cartan, Universit\'e Nancy 1, BP 239, 54506 Vand\oe uvre-l\`es-Nancy, France}
\email{frobert@iecn.u-nancy.fr}
\date{March 11th 2010}
\begin{document}


\maketitle

\centerline{\large\it Dedicated to Jean-Pierre Gossez on the occasion of his 65th birthday}

\section{Introduction and statement of the main result}
Let $M$ be a compact manifold of dimension $n\geq 3$ and let $k\geq 1$ be an integer such that $k\leq \frac{n}{2}$ if $n$ is even. In their celebrated work, Graham-Jenne-Mason-Sparling \cite{gjms} provided a systematic construction of conformally invariant operators (GJMS operators for short) based on the ambient metric of Fefferman-Graham \cite{fg1,fg2}. More precisely, letting ${\mathcal M}$ be the set of Riemannian metrics on $M$, then for all $g\in {\mathcal M}$, there exists an operator $P_g: C^\infty(M)\to C^\infty(M)$ such that

\smallskip(i) $P_g$ is a differential operator and $P_g=\Delta_g^k+lot$

\smallskip (ii) $P_g$ is natural, that is $\varphi^\star P_g=P_{\varphi^\star g}$ for all smooth diffeomorphism $\varphi:M\to M$.

\smallskip (iii) $P_g$ is self-adjoint with respect to the $L^2-$scalar product

\smallskip (iv) Given $\omega\in C^\infty(M)$ and defining $\hat{g}=e^{2\omega}g$, we have that
\begin{equation}\label{eq:conf:e}
P_{\hat{g}}(f)=e^{-\frac{n+2k}{2}\omega}P_g\left(e^{\frac{n-2k}{2}\omega}f\right)\hbox{ for all }f\in C^\infty(M).
\end{equation}

\medskip\noindent Here $\Delta_g:=-\hbox{div}_g(\nabla)$ is the Laplace-Beltrami operator and $lot$ denotes differential terms of lower order. Point (iii) above is due to Graham-Zworski \cite{grzw}. For instance, on $\rn$ endowed with its Euclidean metric $\xi$, one has that $P_\xi=\Delta_\xi^k$. There is a natural scalar invariant, namely the $Q-$curvature, attached to the operator $P_g$: this scalar invariant, denoted as $Q_g$, was initially introduced by Branson and \O rsted  \cite{bo} for $n=2k=4$ and generalized by Branson \cite{br1,br2}. When $k=1$, the GJMS operator is the conformal Laplacian and the $Q-$curvature is the scalar curvature (up to a dimensional constant). When $k=2$, the GJMS operator is the Paneitz operator introduced in \cite{paneitz}.  When $n\neq 2k$, the $Q-$curvature is $Q_g:=\frac{2}{n-2k}P_g(1)$: when $n=2k$, the definition is much more subtle and involves a continuation in dimension argument (we refer to the survey Branson-Gover \cite{BrGo} and to Juhl \cite{juhlbook} for an exposition in book form). In the spirit of classical problems in conformal geometry, our objective here is to prescribe the $Q-$curvature in a conformal class; that is, given a conformal Riemannian class ${\mathcal C}$ on $M$ and a function $f\in C^\infty(M)$, we investigate the existence of a metric $g\in {\mathcal C}$ such that $Q_g=f$. As one checks (see Proposition \ref{prop:crit:pt} below), up to multiplication by a constant, this amounts to finding critical points of the perturbation of the Hilbert functional
$$\begin{array}{lll}
{\mathcal C} & \to &\rr\\
g & \mapsto & \displaystyle{\frac{\int_M Q_g\, dv_g}{V_f(M,g)^{\frac{n-2k}{n}}}}
\end{array}$$
where $V_f(M,g):=\int_M f\, dv_g$ is the weighted $f-$volume of $(M,g)$.  This structure suggests to apply variational methods to prescribe the $Q-$curvature and we define
$$\mu_f({\mathcal C}):=\inf_{g\in {\mathcal C}}\frac{\int_M Q_g\, dv_g}{V_f(M,g)^{\frac{n-2k}{n}}}.$$
Given a metric $g\in {\mathcal C}$, the conformal class can be described as
$${\mathcal C}=\{e^{2\omega}g/\, \omega\in C^\infty(M)\}.$$
We assume that $n>2k$: in this context, it is more convenient to write a metric $\hat{g}\in {\mathcal C}$ as $\hat{g}=u^{\frac{4}{n-2k}}g$ with $u\in C^\infty_+(M)$, the set of positive smooth functions. With this parametrization, we have that 
$${\mathcal C}=\{u^{\frac{4}{n-2k}}g/\, u\in  C^\infty_{+}(M)\},$$
and the relation \eqref{eq:conf:e} between $P_g$ and $P_{\hat{g}}$ rewrites
\begin{equation}\label{eq:conf}
P_{\hat{g}}\varphi=u^{1-\crit}P_g(u\varphi)
\end{equation}
for all $\varphi\in C^\infty(M)$, where $\crit:=\frac{2n}{n-2k}$. Therefore, taking $\varphi\equiv 1$, we have that $$P_g u=\frac{n-2k}{2}Q_{\hat{g}}u^{\crit-1}\hbox{ in }M$$
where $\hat{g}=u^{\frac{4}{n-2k}}g$, and then finding a metric in ${\mathcal C}$ with $f$ as $Q-$curvature amounts to solving the variational elliptic equation $P_g u=\frac{n-2k}{2}f u^{\crit-1}$.  Despite this elegant variational structure, this question gives rise to a crucial intrinsic difficulty due to the essence of the problem, that is the conformal invariance of the operator. More precisely, in the spirit of Bourguignon-Ezin \cite{be}, Delano\"e and the author proved in \cite{dr} that 
$$\int_M X(Q_g)\, dv_g=0$$
for all conformal Killing field $X$ on $(M, {\mathcal C})$. When $k=1$, this is the celebrated Kazdan-Warner obstruction \cite{kw} to the scalar curvature problem. In particular, if $\varphi\in C^\infty(\sn)\setminus \{0\}$ is a first eigenfunction of the Laplace-Beltrami operator on the standard sphere $(\sn,h)$, then for any $\epsilon\neq 0$, $Q_h+\epsilon\varphi$ is not achived as the $Q-$curvature of a metric in the conformal class of the standard sphere. Therefore, a function can be arbitrarily close to a $Q-$curvature but not be a $Q-$curvature itself: the prescription of the $Q-$curvature is then a highly unstable problem, and its underlying analysis is intricate. We refer again to \cite{dr} for considerations on the structure of the set of $Q-$curvatures. In the case $k=1$ and $n\geq 3$, the problem of prescribing a constant $Q-$curvature is known as the Yamabe problem: it is not the purpose of the present article to make an extensive historical review of the famous resolution of this problem, and we refer to Lee-Parker \cite{lp} and the references therein. Concerning fourth order problems, that is for $k=2$, there has been an intensive litterature on the question: here, we refer to the recent surveys of Branson-Gover \cite{BrGo}, Chang \cite{chang}, Malchiodi \cite{m} and the references therein.

\medskip\noindent In the sequel, we will say that a function is admissible if it can be achieved as the $Q-$curvature of a metric in a given conformal class. As seen above, some functions on the sphere are not admissible for the standard conformal class. Moser \cite{moser} had the idea to impose invariance under a group of isometries to find admissible functions on the sphere for the scalar curvature problem in 2D. This strategy was also used by Escobar-Schoen \cite{es} and Hebey \cite{hebeysphere} in higher dimensions. In the same spirit, Delano\"e and the author \cite{dr} proved that a function on the sphere which is close to $Q_h$ and invariant under a group of isometries acting without fixed point is admissible. In the present article, we relax the condition of being close to $Q_h$ by imposing cancelation of some derivatives (see Theorem \ref{thm:n} below). In the specific case $n=2k+1$, very few is required; this is the object our main result:

\begin{thm}\label{th:main} Let $k\geq 1$ and let $G$ be a subgroup of isometries of $(\mathbb{S}^{2k+1},h)$. Let $f\in C^\infty(M)$ be a positive $G-$invariant function and assume that $G$ acts without fixed point (that is $|O_G(x)|\geq 2$ for all $x\in \mathbb{S}^{2k+1}$). Then there exists $g\in [h]$ such that $Q_g=f$ and $G\subset Isom_g(\sn)$. 
\end{thm}
When $k=1,2$, this result is due respectively to Hebey \cite{hebeysphere} and to the author \cite{robertpams}. This theorem is a particular case of more general results proved on arbitrary conformal manifolds (see Proposition \ref{prop:strict} and Theorem \ref{thm:n} below). In this article, we make a general analysis of the operator $P_g$ and of the blow-up phenomenon attached to it on arbitrary conformal manifolds. In the last section, we apply this analysis to the conformal sphere.

\medskip\noindent {\bf Acknowledgement:} This work was initiated when the author was visiting the Technische Universit\"at in Berlin supported by an Elie Cartan followship from the Stiftung Luftbr\"uckendank. It is a pleasure to thank the Differential Geometry team in TU, in particular Mike Scherfner, and the Stiftung for their support and kind hospitality. The author also thanks Andreas Juhl for fruitful discussions on this work.

\section{Miscellaneous on the operator $P_g$}
The operator $P_g$ can be written (partially) as a divergence form (we refer to Branson-Gover \cite{BrGo}):  as a preliminary step, we precise this divergence form that will be useful in the sequel:
\begin{prop}\label{prop:selfadj} Let $P_g$ be the conformal GJMS operator. Then for any $l\in\{0,...,k-1\}$, there exists $A_{(l)}(g)$ a smooth $T_{2l}^0-$tensor field on $M$ such that
\begin{equation}\label{eq:Pg:div}
P_g=\Delta_g^k+\sum_{l=0}^{k-1}(-1)^l \nabla^{j_l...j_1}(A_{(l)}(g)_{i_1...i_lj_1...j_l}\nabla^{i_1...i_l}),
\end{equation}
where the indices are raised via the musical isomorphism. In addition for any $l\in\{0,...,k-1\}$, $A_{(l)}(g)$ is symmetric in the following sense: $A_{(l)}(g)(X,Y)=A_{(l)}(g)(Y,X)$ for all $X,Y$  $T_0^{l}-$tensors on $M$. In particular, we have that
\begin{equation}\label{eq:Pg:adj}
\int_M u P_g(v)\, dv_g=\int_M\left(\Delta_g^{\frac{k}{2}}u\Delta_g^{\frac{k}{2}}v +\sum_{l=0}^{k-1}A_{(l)}(g)(\nabla^l u,\nabla^l v)\right)\, dv_g
\end{equation} 
for all $u, v\in C^\infty(M)$. Here, we have adopted the convention 
$$\Delta_g^{\frac{k}{2}}u\Delta_g^{\frac{k}{2}}v:=(\nabla \Delta_g^{\frac{k-1}{2}}u,\nabla\Delta_g^{\frac{k-1}{2}}v)_g$$
when $k$ is odd.
\end{prop}
\begin{proof} The proof uses only the self-adjointness of the operator $P_g$. In the sequel, we note $A^\star$ the adjoint of any operator $A$ with respect to the $L^2-$product. As a preliminary, we compute the adjoint of some elementary operators. We adopt here Hamilton's convention \cite{hamilton}: the notation $A\star B$ denotes a linear combination of contraction of the tensors $A$, $B$, $g$ and $g^{-1}$. Given $B$ a smooth $T_q^0-$tensor field on $M$, we consider the operator $Bu:=B\cdot\nabla^qu=B_{i_1...i_q}\nabla^{i_1...i_q}u$ for all $u\in C^\infty(M)$. 

\smallskip\noindent We claim that 
$$B^\star=(-1)^qB+\sum_{l=1}^{q-1} \nabla^{l}u\star \nabla^{q-l}B.$$
We prove the claim. We let $u, v\in C^\infty(M)$ be two smooth functions on $M$. Integrating by parts, we have that
\begin{eqnarray*}
\int_M u Bv\, dv_g &=& \int_M u B_{i_1...i_q}\nabla^{i_1...i_q}v\, dv_g=(-1)^q\int_M \nabla^{i_q...i_1}(u B_{i_1...i_q})v\, dv_g\\
&=& (-1)^q\int_M \left(B_{i_1...i_q}\nabla^{i_q...i_1} u +\sum_{l=0}^{q-1} \nabla^{l}u\star \nabla^{q-l}B\right)v\, dv_g.
\end{eqnarray*}
Therefore, $B^\star$ is defined and
$$B^\star u=(-1)^qB_{i_1...i_q}\nabla^{i_q...i_1} u +\sum_{l=0}^{q-1} \nabla^{l}u\star \nabla^{q-l}B.$$
For any smooth tensor field $T$, we define $Asym(T)(X,Y,...):=T(X,Y,...)-T(Y,X,...)$. It follows from the definition of the curvature tensor that 
\begin{equation*}
Asym(\nabla^2T)=T\star R,
\end{equation*}
where $R$ is the curvature tensor. Therefore, for any permutation $\sigma$ of $\{1,...,q\}$, we have that
\begin{equation}\label{action:sigma}
\nabla^q u-\sigma\cdot\nabla^q u=\nabla^{q-2}u\star R,
\end{equation}
where $\sigma\cdot T$ permutes the variables of the covariant tensor $T$ along $\sigma$. Therefore,  we have that $\nabla^{i_q...i_1}u-\nabla^{i_1...i_q}u$ is a contraction of $\nabla^{q-2}u$, and therefore we get that $B^\star=(-1)^q B+lot$. This proves the claim.

\medskip\noindent We are now in position to prove Proposition \ref{prop:selfadj}. It follows from the definition of $P_g$ that there exists $B$, a smooth $T_{2k-1}^0-$tensor field on $M$, such that
$P_gu=\Delta_g^k u +B u+lot$
for all $u\in C^\infty(M)$. Since $P_g$ and $\Delta_g$ are self-adjoint, we then get that
$$P_g=P_g^\star=\Delta_g^k+B^\star+lot=\Delta_g^k-B+lot$$
since $2k-1$ is odd. In particular, $Bu=lot$ and therefore, $Bu=0$ for all $u\in C^\infty(M)$. 

\medskip\noindent We now take $C$ a smooth $(2k-2,0)-$tensor field such that $P_g=\Delta_g^k +C\cdot\nabla^{2k-2}+lot$. We define $A$ as the symmetrized tensor of $C$, that is via coordinates $A(X,Y)=(-1)^{k-1}\frac{1}{2}(C(X,Y)+C(Y,X))$ for all $X,Y$ any $T_0^{k-1}-$tensors on $M$. As easily checked, since changing the order of differentiation involves only lower order terms via with \eqref{action:sigma}, we have that
\begin{eqnarray*}
C\cdot\nabla^{2k-2}u&=&C_{i_1...i_{k-1}j_1...j_{k-1}}\nabla^{i_1...i_{k-1}j_1...j_{k_1}}u\\
&=& (-1)^{k-1}A_{i_1...i_{k-1}j_1...j_{k-1}}\nabla^{i_1...i_{k-1}j_1...j_{k-1}}u+\nabla^{2k-4}u\star R\\
&=& (-1)^{k-1}A_{i_1...i_{k-1}j_1...j_{k-1}}\nabla^{j_{k-1}...j_{1}i_1...i_{k-1}}u+\nabla^{2k-4}u\star R\\
&=& (-1)^{k-1}\nabla^{j_{k-1}...j_{1}}\left(A_{i_1...i_{k-1}j_1...j_{k-1}}\nabla^{i_1...i_{k-1}}u\right)\\
&&+\nabla^{2k-4}u\star R+\sum_{l=1}^{k-1} \nabla^{2k-2-l}u\star\nabla^{l}A
\end{eqnarray*}
and then 
$$P_g=\Delta_g^k+(-1)^{k-1}\nabla^{j_{k-1}...j_{1}}\left(A_{i_1...i_{k-1}j_1...j_{k-1}}\nabla^{i_1...i_{k-1}}\right)+lot.$$
Iterating these steps yields \eqref{eq:Pg:div}. Integrating by parts then yields \eqref{eq:Pg:adj}.\end{proof}

\medskip\noindent Define the norm $\Vert u\Vert_{H_k^2}:=\sum_{l=0}^k\Vert\nabla^l u\Vert_2$ and the space $H_k^2(M)$ as the completion of $C^\infty(M)$ for the norm $\Vert\cdot\Vert_{H_k^2}$. As a consequence of \eqref{eq:Pg:adj}, we get that the bilinear form $(u, v)\mapsto \int_M u P_gv\, dv_g$ extends to a continuous symmetrical bilinear form on $H_k^2(M)\times H_k^2(M)$. We say that $P_g$ is coercive if there exists $c>0$ such that
$$\int_M uP_g u\, dv_g\geq c\Vert u\Vert_2^2\hbox{ for all }u\in H_k^2(M).$$
We then define the norm $\Vert u\Vert_{P_g}:=\sqrt{\int_M uP_g u\, dv_g}$ for all $u\in H_k^2(M)$.

\begin{prop}\label{prop:norm} Assume that $P_g$ is coercive. Then $\Vert\cdot\Vert_{P_g}$ is a norm on $H_k^2$ equivalent to $\Vert\cdot\Vert_{H_k^2}$. 
\end{prop}
\begin{proof} Clearly $\Vert\cdot\Vert_{P_g}$ is a norm and there exists $C>0$ such that $\Vert\cdot\Vert_{P_g}\leq C\Vert\cdot\Vert_{H_k^2}$. We now argue by contradiction and we assume that the two norms are not equivalent: then there exists $(u_i)_{i\in\nn}\in H_k^2(M)$ such that
\begin{equation}\label{hyp:ui}
\Vert u_i\Vert_{H_k^2}=1\hbox{ and }\Vert u_i\Vert_{P_g}=o(1)
\end{equation}
when $i\to +\infty$. Up to a subsequence, still denoted as $(u_i)$, there exists $u\in H_k^2(M)$ such that $u_i\rightharpoonup u$ weakly in $H_k^2(M)$ and $u_i\to u$ strongly in $H_{k-1}^2(M)$ when $i\to +\infty$. The coercivity of $P_g$ yields $\Vert u_i\Vert_2=o(1)$ when $i\to +\infty$, and then $u\equiv 0$. Therefore, we have that
\begin{equation}\label{lim:ui}
u_i\rightharpoonup 0\hbox{ weakly in }H_k^2(M)\hbox{ and }u_i\to 0\hbox{ strongly in }H_{k-1}^2(M)
\end{equation}
when $i\to +\infty$. Consequently, \eqref{hyp:ui} rewrites
\begin{equation}\label{hyp:ui:0}
\lim_{i\to +\infty}\int_M|\nabla^k u_i|_g^2\, dv_g=1\hbox{ and }\lim_{i\to +\infty}\int_M(\Delta_g^{\frac{k}{2}}u_i)^2\, dv_g=0.
\end{equation}
The contradiction comes from a Bochner-Lichnerowicz-Weitzenbock type formula. Here again, we use \eqref{action:sigma}. We fix $u, v\in C^\infty(M)$: we have that (the notation $a\equiv b$ means that the terms are equal up to a divergence)
\begin{eqnarray*}
(\nabla^k u,\nabla^k v)_g&\equiv& g^{\alpha_1\beta_1}...g^{\alpha_k\beta_k}\nabla_{\alpha_1...\alpha_k}u\nabla_{\beta_1...\beta_k}v\\
&\equiv& -g^{\alpha_1\beta_1}...g^{\alpha_k\beta_k}\nabla_{\beta_1\alpha_1...\alpha_k}u\nabla_{\beta_2...\beta_k}v\\
&\equiv& -g^{\alpha_1\beta_1}...g^{\alpha_k\beta_k}\nabla_{\alpha_2...\alpha_k\beta_1\alpha_1}u\nabla_{\beta_2...\beta_k}v+\nabla^{k-1}u\star \nabla^{k-1}v\star R\\
&\equiv& -g^{\alpha_2\beta_2}...g^{\alpha_k\beta_k}\nabla_{\alpha_2...\alpha_k}g^{\alpha_1\beta_1}\nabla_{\beta_1\alpha_1}u\nabla_{\beta_2...\beta_k}v+\nabla^{k-1}u\star\nabla^{k-1}v\\
&\equiv& g^{\alpha_2\beta_2}...g^{\alpha_k\beta_k}\nabla_{\alpha_2...\alpha_k}\Delta_g u\nabla_{\beta_2...\beta_k}v+\nabla^{k-1}u\star\nabla^{k-1}v\star R\\
&\equiv& (\nabla^{k-1} \Delta_g u,\nabla^{k-1}v)_g+\nabla^{k-1}u\star\nabla^{k-1}v\star R.
\end{eqnarray*}
the same procedure applied to $(\nabla^{k-1}v,\nabla^{k-1}\Delta_g u)_g$ yields
\begin{eqnarray*}
(\nabla^k u,\nabla^k v)_g&\equiv& (\nabla^{k-2} \Delta_g u,\nabla^{k-2}\Delta_g v)_g\\
&&+\nabla^{k-1}u\star\nabla^{k-1}v\star R
+\nabla^{k-2}\Delta_g u\star\nabla^{k-2}v\star R.
\end{eqnarray*}
Taking $u=v=u_i$, integrating over $M$ and using \eqref{lim:ui} yields
$$\int_M |\nabla^k u_i|_g^2\, dv_g=\int_M |\nabla^{k-2}\Delta_g u_i|_g^2\, dv_g+o(1)$$
when $i\to +\infty$. Iterating this process and considering separately the cases $k$ odd and $k$ even, we get that 
$$\int_M |\nabla^k u_i|_g^2\, dv_g=\int_M (\Delta_g^{\frac{k}{2}} u_i)^2\, dv_g+o(1)$$
when $i\to +\infty$. This is a contradiction with \eqref{hyp:ui:0} and Proposition \ref{prop:norm} is proved.\end{proof}

\section{General considerations on the equivariant Yamabe invariant}
We let $(M,{\mathcal C})$ be a conformal Riemannian manifold. We let $G\subset Diff(M)$ be a subgroup of diffeomorphisms of $M$. We define
$${\mathcal C}_G:=\{g\in {\mathcal C}/\, G\subset Isom_g(M)\},$$
and we assume that ${\mathcal C}_G\neq\emptyset$. As easily checked, for any $g\in {\mathcal C}_G$, we have that
$${\mathcal C}_G=\{e^{2\omega}g/\, \omega\in C^\infty_G(M)\}$$
where $C^\infty_G(\Omega)=\{\omega\in C^\infty(M)/\, \omega\circ\sigma=\omega\hbox{ for all }\sigma\in G\}$ is the set of $G-$invariant smooth functions on $M$. We assume that $n>2k$: in this context, it is more convenient to write a metric $\hat{g}\in {\mathcal C}$ as $\hat{g}=u^{\frac{4}{n-2k}}g$ with $u\in C^\infty_+(M)$. The relation between $P_g$ and $ P_{\hat{g}}$ is given by \eqref{eq:conf}. With the new parametrization, we have that 
$${\mathcal C}_G=\{u^{\frac{4}{n-2k}}g/\, u\in  C^\infty_{G,+}(M)\},$$
where $C_{G,+}^\infty(M):=\{u\in C_G^\infty(M)/\; u>0\}$. Let $f\in C_{G,+}^\infty(M)$ be a smooth positive $G-$invariant function. By analogy with the Yamabe invariant, we define
\begin{equation*}
\mu_f({\mathcal C}_G):=\inf_{g\in {\mathcal C}_G}\frac{\int_M Q_g\, dv_g}{V_f(M,g)^{\frac{2}{\crit}}}
\end{equation*}
where $V_f(M,g)$ is the $f-$volume defined in the introduction and $\crit:=\frac{2n}{n-2k}$. We fix $g\in {\mathcal C}_G$: as easily checked, we have that
\begin{equation*}
\mu_f({\mathcal C}_G)=\frac{2}{n-2k}\inf_{u\in C^\infty_{G,+}(M)}I_g(u)
\end{equation*}
where 
$$I_g(u):=\frac{\int_M u P_g u\, dv_g}{\left(\int_M f |u|^{\crit}\, dv_g\right)^{\frac{2}{\crit}}}$$
for all $u\in H_k^2(M)\setminus\{0\}$.

\begin{prop}\label{prop:crit:pt} A metric $g\in{\mathcal C}_G$ is a critical point of the functional $g\mapsto \frac{\int_M Q_g\, dv_g}{V_f(M,g)^{\frac{2}{\crit}}}$ if and only if there exists $\lambda\in\rr$ such that $Q_g=\lambda f$.
\end{prop}
\begin{proof} We fix $g\in {\mathcal C}_G$ and $t\mapsto g(t)\in {\mathcal C}_G$ a differentiable family of metrics conformal to $g$ such that $g(0)=g$. In particular, there exists a differentiable family $t\mapsto u(t)\in C^\infty_{G, +}(M)$ such that $g(t)=u(t)^{\frac{4}{n-2k}}g$ and $u(0)=1$. We define $\dot{u}:=u'(0)$. Using the self-adjointness of $P_g$, straightforward computations yield
$$\frac{d}{dt}\left(\frac{\int_M Q_{g(t)}\, dv_{g(t)}}{V_f(M,g(t))^{\frac{2}{\crit}}}\right)_{t=0}=2\frac{\int_M \dot{u}\left(Q_g-f \bar{Q}_g^f\right)\, dv_g}{V_f(M, g(t))^{\frac{2}{\crit}}}$$
where
$$\bar{Q}_g^f=\frac{\int_M Q_g\, dv_g}{V_f(M,g)}.$$
Since $u$ is $G-$invariant, the function $\dot{u}$ ranges $C^\infty_G(M)$. Fix $v\in C^\infty(M)$ and let $v_G$ be its symmetrization via the Haar measure. We then define $u(t):=1+tv_G$ for all $t\in\rr$: since $f$ and $Q_g$ are $G-$invariant (this is a consequence of point (ii) of the characterization of $P_g$ and of the definition of $Q_g$), we get that
$$\int_M \dot{u}\left(Q_g-f \bar{Q}_g^f\right)\, dv_g=\int_M v_G\left(Q_g-f \bar{Q}_g^f\right)\, dv_g=\int_M v\left(Q_g-f \bar{Q}_g^f\right)\, dv_g.$$
Therefore, $g$ is a critical point if and only if $Q_g=f \bar{Q}_g^f$. This proves Proposition \ref{prop:crit:pt}.\end{proof}

\medskip\noindent To carry out the analysis, coercivity and positivity preserving property are required. More precisely, we assume that there exists $g\in {\mathcal C}$ such that

$$\left\{\begin{array}{cl}
(C) & \hbox{ the operator }P_g\hbox{ is coercive}\\
(PPP) & \hbox{ for any }u\in C^\infty(M)\hbox{ such that }P_g\geq 0\hbox{ then either }u>0\hbox{ or }u\equiv 0
\end{array}\right\}.$$
\smallskip\noindent Note that $(C)$ and $(PPP)$ are conformally invariant: they hold for some $g\in{\mathcal C}$ iff they hold for all $g\in{\mathcal C}$. 

\begin{prop}\label{prop:ppp} Assume that the metric $g$ is Einstein with positive scalar curvature and $n>2k$, then $P_g$ satisfies $(C)$ and $(PPP)$.
\end{prop}
\begin{proof} This relies essentially on the the explicit expression of the GJMS operator in the Einstein case: see Proposition 7.9 of Fefferman-Graham \cite{fg2} and also Gover \cite{gover} for a proof via tractors. Indeed, for an Einstein metric $g$, $P_g$ expresses as an explicit product of second-order operators with constant coefficients depending only on the scalar curvature. For positive curvature, a direct consequence is that $P_g$ satisfies $(PPP)$ by $k$ applications of the second-order comparison principle.  Moreover, still in this case, since $P_g=S(\Delta_g)$ with $S$ a polynomial with positive constant coefficients, it follows from Hebey-Robert \cite{hr} that the first eigenvalue of $P_g$ is $S(0)>0$ ($0$ is the first eigenvalue of $\Delta_g$), and then $P_g$ satisfies $(C)$.\end{proof}

\medskip\noindent Due to the lack of compactness of the embedding $H_k^2(M)\hookrightarrow L^{\crit}(M)$, it is standard to use the subcritical method. Given $q\in (2,\crit]$, we define
$$I_{g,q}(u):=\frac{\int_M u P_g u\, dv_g}{\left(\int_M f |u|^{q}\, dv_g\right)^{\frac{2}{q}}}$$
for all $u\in H_k^2(M)\setminus\{0\}$, and
$$\mu_q:=\inf_{u\in H_{k,G}^2(M)\setminus\{0\}}I_{g,q}(u),$$
where $H_{k,G}^2(M):=\{u\in H_k^2(M)/\, u\circ\sigma=u\hbox{ a.e. for all }\sigma\in G\}$. The first result is that $\mu_q$ is achieved at a smooth positive minimizer when $q<\crit$:

\begin{prop}\label{prop:ss:crit} We fix $q\in (2,\crit)$, we assume that $(C)$ and $(PPP)$ hold and that ${\mathcal C}_G\neq\emptyset$. Then $\mu_q>0$ is achieved. Moreover, there exists $u_q\in C^\infty_{G,+}(M)$ a smooth positive function such that $\mu_q=I_{g,q}(u_q)$ and
\begin{equation}\label{eq:uq}
P_g u_q=\mu_q f u_q^{q-1}\hbox{ in }M\hbox{ with }\int_M f u_q^q\, dv_g=1.
\end{equation}
\end{prop}
\begin{proof} Since $P_g$ is coercive, the norms $\Vert\cdot\Vert_{H_k^2}$ and $\Vert\cdot\Vert_{P_g}$ are equivalent, and then, it follows from H\"older's and Sobolev's inequality that
\begin{eqnarray}
&&\left(\int_M f|u|^q\, dv_g\right)^{\frac{2}{q}}\leq \left(\int_M f\, dv_g\right)^{\frac{2}{q}-\frac{2}{\crit}}\left(\int_M f|u|^{\crit}\, dv_g\right)^{\frac{2}{\crit}}\label{ineq:holder}\\
&&\leq C\left(\int_M f\, dv_g\right)^{\frac{2}{q}-\frac{2}{\crit}}\Vert u\Vert_{H_k^2}^2\leq C'\left(\int_M f\, dv_g\right)^{\frac{2}{q}-\frac{2}{\crit}}\Vert u\Vert_{P_g}^2,\nonumber
\end{eqnarray}
and then $I_{g,q}(u)\geq (C')^{-1}\left(\int_M f\, dv_g\right)^{-\frac{2}{q}+\frac{2}{\crit}}$ for all $u\in H_k^2(M)\setminus\{0\}$, and therefore $\mu_q>0$. The existence of a minimizer is standard and we omit it. Let us take then $u\in H_{k,G}^2(M)\setminus\{0\}$ be a mimimizer. Without loss of generality, we can assume that $\int_M f |u|^q\, dv_{g}=1$.


\medskip\noindent The Euler-Lagrange equation for $I_{g,q}$ yields $I_{g, q}'(u)\varphi=0$ for all $\varphi\in H_{k,G}^2(M)$. Using the Haar measure and arguing as in the proof of Proposition \ref{prop:crit:pt} (see also \cite{hebeysphere}), we get that this equality holds for all $\varphi\in H_k^2(M)$. Since the exponent $q$ is subcritical, we get with standard bootstrap arguments that $u\in C^{2k}_G(M)$ and $P_g u=\mu_q f |u|^{q-2}u$. We are left with proving that $u>0$ or $u<0$. We let $v\in C^{2k}_G(M)$ be such that $P_g v=|P_g u|$ in $M$. Since $u\not\equiv 0$, it follows from $(PPP)$ that $v\geq |u|$ and $v>0$.  Using again the definition of $\mu_q$, we have that
\begin{eqnarray*}
\mu_q &\leq & \frac{\int_M v P_g v\, dv_g}{\left(\int_M f v^q\, dv_g\right)^{\frac{2}{q}}}=\mu_q\frac{\int_M f v |u|^{q-1}\, dv_g}{\left(\int_M f v^q\, dv_g\right)^{\frac{2}{q}}}\\
&\leq &  \mu_q \frac{\left(\int_M f v^q\, dv_g\right)^{\frac{1}{q}}\left(\int_M f |u|^q\, dv_g\right)^{\frac{q-1}{q}}}{\left(\int_M f v^q\, dv_g\right)^{\frac{2}{q}}}\\
&\leq & \mu_q\left(\int_M f |u|^q\, dv_g\right)^{\frac{q-2}{q}}=\mu_q\hbox{ since }v\geq |u|
\end{eqnarray*}
Therefore equality holds everywhere and $|u|=v>0$. In particular $u$ does not change sign, and we can assume that it is positive. Bootstrap and regularity theory (see \cite{adn})  then yield $u\in C^\infty_{G,+}(M)$, and Proposition \ref{prop:ss:crit} is proved with $u_q:=u$.\end{proof}

\begin{prop}\label{prop:lim:muq} We claim that $\lim_{q\to \crit}\mu_q=\mu_{\crit}=\frac{n-2k}{2}\mu_f({\mathcal C}_G)$.
\end{prop}
\begin{proof} Using the H\"older's inequality \eqref{ineq:holder}, we get that $$I_{g, \crit}(u)\leq I_{g,q}(u) V_f(M,g)^{\frac{2}{q}-\frac{2}{\crit}}$$ for all $u\in H_k^2(M)\setminus \{0\}$, and then $\mu_{\crit}\leq \mu_q V_f(M,g)^{\frac{2}{q}-\frac{2}{\crit}}$, which yields $\mu_{\crit}\leq \liminf_{q\to\crit}\mu_q$. Conversely, fix $\epsilon>0$ and let $u\in H_{k,G}^2(M)\setminus\{0\}$ be such that $I_{g, \crit}(u)< \mu_{\crit}+\epsilon$. Since $\lim_{q\to\crit}I_{g,q}(u)=I_{g, \crit}(u)$, we then get that there exists $q_0<\crit$ such that $\mu_q<\mu_{\crit}+\epsilon$ for $q\in (q_0, \crit)$, and then $\limsup_{q\to\crit}\mu_q\leq \mu_{\crit}$. Therefore, $\lim_{q\to \crit}\mu_q=\mu_{\crit}$.

\medskip\noindent For $q\in (2,\crit]$, we define $\mu_{q,+}:=\inf\{I_{g,q}(u)/\, u\in H_{k,G}^2(M)\setminus\{0\}\hbox{ and }u\geq 0\hbox{ a.e.}\}$. Arguing as above, we get that $\lim_{q\to \crit}\mu_{q,+}=\mu_{\crit,+}$. Since $\mu_{q,+}=\mu_q$ for all $q<\crit$ with Proposition \ref{prop:ss:crit}, we then get that $\mu_{\crit}=\mu_{\crit,+}$.

\medskip\noindent We claim that $\mu_{\crit,+}=\frac{n-2k}{2}\mu_f({\mathcal C}_G)$. Indeed, via local convolutions with a positive kernel, we get that $C^\infty_{+}(M)$ is dense in $H_{k,+}^2(M)$ for the $H_k^2-$norm. A symmetrization via the Haar measure then yields that $C^\infty_{G,+}(M)$ is dense in $H_{k, G,+}^2(M)$: clearly this yields  $\mu_{\crit,+}=\frac{n-2k}{2}\mu_f({\mathcal C}_G)$, and the claim is proved.
\end{proof}

\medskip\noindent We define $D_k^2(\rn)$ as the completion of $C^\infty_c(\rn)$ for the norm $u\mapsto \Vert \Delta_\xi^{\frac{k}{2}}u\Vert_2$ and we define
\begin{equation}\label{def:K}
\frac{1}{K(n,k)}:=\inf_{u\in D_k^2(\rn)\setminus\{0\}}\frac{\int_{\rn}(\Delta_\xi^{\frac{k}{2}}u)^2\, dv_\xi}{\left(\int_{\rn}|u|^{\crit}\, dv_\xi\right)^{\frac{2}{\crit}}}.
\end{equation}
It follows from Sobolev's embedding theorem that $K(n,k)>0$. Moreover, it follows from Lions \cite{lions} that the infimum is achieved by $U:x\mapsto (1+|x|^2)^{k-\frac{n}{2}}$, and that all minimizers are compositions of $U$ by translations and homotheties.

\begin{prop}\label{prop:large} We have that
\begin{equation}\label{ineq:large}
\mu_f({\mathcal C}_G)\leq \frac{2}{n-2k}\cdot\frac{|O_G(x)|^{\frac{2k}{n}}}{f(x)^{\frac{2}{\crit}}K(n,k)}
\end{equation}
for all $x\in M$, where $|O_G(x)|$ denotes the cardinal (possibly $\infty$) of the orbit $O_G(x)$.
\end{prop}
\begin{proof} We fix $x\in M$. Without loss of generality, we assume that $m:=|O_G(x)|<+\infty$ (otherwise \eqref{ineq:large} is clear). We let $\sigma_1=Id_M,...,\sigma_m\in G$ be such that $O_G(x)=\{x_1,...,x_m\}$ where $x_i=\sigma_i(x)$ for all $i\in \{1,...,m\}$ are distinct. We let $u\in C^\infty_c(\rn)$ be a radially symmetrical smooth function and we define for $\epsilon>0$ small the function
$$u_{\epsilon, i}(z):=u\left(\frac{1}{\epsilon}\hbox{exp}_{x_i}^{-1}(z)\right)\hbox{ if }d_g(z, x_i)<i_g(M)\hbox{ and }0\hbox{ otherwise.}$$
Clearly, $u_{\epsilon,i}\in C^\infty(M)$ for $\epsilon>0$ small enough. We now define
$$u_\epsilon:=\sum_{i=1}^mu_{\epsilon,i}.$$
As one checks, since $u$ is radially symmetrical, we have that $u_\epsilon\in C^\infty_G(M)$ is $G-$invariant for $\epsilon>0$ small enough. 

\medskip\noindent Let us compute $I_{g, \crit}(u_\epsilon)$. We fix $\delta\in (0, i_g(M))$ and we define the metric $g_\epsilon:=(\hbox{exp}^\star_g)(\epsilon\cdot)$: since the elements of $G$ are isometries (and then $P_g=P_{\sigma^\star g}=\sigma^\star P_g$ for all $\sigma\in G$) and the $u_{\epsilon,i}$'s have disjoint supports, we get that
\begin{eqnarray*}
\int_M u_\epsilon P_g u_\epsilon\, dv_g&=& \sum_{i,j=1}^m \int_M u_{\epsilon, i}P_gu_{\epsilon, j}\, dv_g=\sum_{i=1}^m \int_M u_{\epsilon, i}P_gu_{\epsilon, i}\, dv_g\\
&=& \sum_{i=1}^m \int_{B_\delta(x_i)} u_{\epsilon, 1}\circ\sigma_{i}^{-1}P_g(u_{\epsilon, 1}\circ\sigma_i^{-1})\, dv_g\\
&=& m\int_{B_\delta(x)} u_{\epsilon, 1}P_g u_{\epsilon, 1}\, dv_g= m\epsilon^{n-2k}\int_{B_{\epsilon^{-1}\delta}(0)} u P_{g_\epsilon}u\, dv_{g_\epsilon}
\end{eqnarray*}
since $\lim_{\epsilon\to 0}g_\epsilon=\xi$, the Euclidean metric, we get that
$$\int_M u_\epsilon P_g u_\epsilon\, dv_g=\epsilon^{n-2k}\left(m\int_{\rn}(\Delta_\xi^{\frac{k}{2}}u)^2\, dv_\xi+o(1)\right)$$
when $\epsilon\to 0$. Similarly, using the $G-$invariance of $f$, we get that
$$\int_M f |u_\epsilon|^{\crit}\, dv_g=\epsilon^{n}\left(mf(x)\int_{\rn}|u|^{\crit}\, dv_\xi+o(1)\right)$$
when $\epsilon\to 0$, and then
$$I_{g, \crit}(u_\epsilon)=\frac{m^{\frac{2k}{n}}}{f(x)^{\frac{2}{\crit}}}\cdot\frac{\int_{\rn}(\Delta_\xi^{\frac{k}{2}}u)^2\, dv_\xi}{\left(\int_{\rn}|u|^{\crit}\, dv_\xi\right)^{\frac{2}{\crit}}}+o(1)$$
when $\epsilon\to 0$. Therefore, since $\mu_f({\mathcal C}_G)=\mu_{\crit}$, taking the limit $\epsilon\to 0$ and taking the infimum on the $u$'s, we get that
$$\mu_{\crit}\leq \frac{|O_G(x)|^{\frac{2k}{n}}}{f(x)^{\frac{2}{\crit}}}\cdot\inf_{u\in C^\infty_c(\rn)\setminus\{0\}\hbox{ radial}}\frac{\int_{\rn}(\Delta_\xi^{\frac{k}{2}}u)^2\, dv_\xi}{\left(\int_{\rn}|u|^{\crit}\, dv_\xi\right)^{\frac{2}{\crit}}}$$
It follows from Lions \cite{lions} that the infimum $K(n,k)^{-1}$ is achieved at smooth radially symmetrical functions, therefore we obtain \eqref{ineq:large}.\end{proof}

\section{The quantization of the formation of singularities}
The objective of this section is to prove the following result:
\begin{thm}\label{th:quanti} Let $(M, {\mathcal C})$ be a conformal Riemannian manifold of dimension $n\geq 3$ and let $k\in \nn^\star$ be such that $2k<n$. Let $G$ be a group of diffeomorphisms such that ${\mathcal C}_G\neq\emptyset$ and let $f\in C^\infty_{G,+}(M)$ be a positive $G-$invariant function. Assume that there exists $g\in {\mathcal C}$ such that $P_g$ satisfies $(C)$ and $(PPP)$. For any $q\in (2,\crit)$, we let $u_q\in C^\infty_{G,+}(M)$ as in Proposition \ref{prop:ss:crit}. Then:

\smallskip(i) either $\limsup_{q\to +\infty}\Vert u_q\Vert_\infty=+\infty$, and there exists $x\in M$ such that $\nabla f(x)=0$ and
$$\mu_f({\mathcal C}_G)=\frac{2}{n-2k}\cdot\frac{|O_G(x)|^{\frac{2k}{n}}}{f(x)^{\frac{2}{\crit}}K(n,k)},$$

\smallskip(ii) or $\Vert u_q\Vert_\infty\leq C$ for all $q<\crit$, and there exists $u\in C^\infty_{G,+}(M)$ such that $\lim_{q\to\crit}u_q=u$ in $C^{2k}(M)$ and $P_g u=\frac{n-2k}{2}\mu_f({\mathcal C}_G) fu^{\crit-1}$ in $M$. In particular, there exists $\hat{g}\in {\mathcal C}_G$ such that $Q_{\hat{g}}=f$ and the infimum $\mu_f({\mathcal C}_G)$ is achieved.

\end{thm}
This type of result is classical. The proof of Theorem \ref{th:quanti} goes through nine steps. For $q\in (2,\crit)$, we let $u_q\in C^\infty_{G,+}(M)$ be as in Proposition \ref{prop:ss:crit} (this is relevant since $(C)$ and $(PPP)$ hold).

\medskip\noindent{\bf Step 1:} We assume that there exists $C>0$ such that  $\Vert u_q\Vert_\infty\leq C$ for all $q<\crit$. We claim that (ii) of Theorem \ref{th:quanti} holds.

\smallskip\noindent We prove the claim. Indeed, it follows from \eqref{eq:uq}, Proposition \ref{prop:lim:muq}, the uniform bound of $(u_q)_q$ in $L^\infty$ and standard elliptic (see for instance \cite{adn}), that, up to a subsequence, there exists $u\in C^{2k}(M)$ nonnegative such that $\lim_{q\to \crit}u_q=u$ in $C^{2k}(M)$: therefore, $P_g u=\mu_{\crit} fu^{\crit-1}$ in $M$ and $\int_M f u^{\crit}\, dv_g=1$. In particular, $P_g u\geq 0$ and $u\not\equiv 0$, and then it follows from $(PPP)$ that $u>0$. Since $u_q$ is $G-$invariant for all $q\in (2,\crit)$, we get that $u\in C^\infty_{G,+}(M)$. Moreover, $I_{g}(u)=\mu_{\crit}$, and then the metric $u^{\frac{4}{n-2k}}g$ is extremal for $\mu_f({\mathcal C}_G)$: it then follows from Proposition \ref{prop:crit:pt} that $\hat{g}:=(\mu_f({\mathcal C}_G))^{1/k}u^{\frac{4}{n-2k}}g$ is also an extremal for $\mu_f({\mathcal C}_G)$ and $Q_{\hat{g}}=f$. This ends Step 1.

\medskip\noindent From now on, we assume that $\limsup_{q\to\crit}\Vert u_q\Vert_\infty=+\infty$. For the sake of clearness, we will write $(u_q)$ even for a subsequence of $(u_q)$. For any $q\in (2,\crit)$, we let $x_q\in M$ be such that
\begin{equation}\label{def:xq}
u_q(x_q)=\max_M u_q\hbox{ and }\lim_{q\to \crit}u_q(x_q)=+\infty.
\end{equation} 
We define 
\begin{equation*}
\alpha_q:=u_q(x_q)^{-\frac{2}{n-2k}}\hbox{ and }\beta_q:=\alpha_q^{\frac{q-2}{\crit-2}}
\end{equation*}
for all $q\in (2, \crit)$. It follows from \eqref{def:xq} that
\begin{equation}\label{lim:alpha}
\lim_{q\to\crit}\alpha_q=0\hbox{ and }\beta_q\geq \alpha_q\hbox{ for }q\to\crit.
\end{equation}
We define
\begin{equation}\label{def:tuq}
\tilde{u}_q(x):=\alpha_q^{\frac{n-2k}{2}}u_q(\hbox{exp}_{x_q}(\beta_q x))
\end{equation}
for all $x\in B_{\beta_q^{-1}\delta}(0)$, where $\delta\in (0, i_g(M))$. 

\medskip\noindent{\bf Step 2:} We claim that there exists $\tilde{u}\in C^{2k}(\rn)$ such that $\lim_{q\to\crit}\tilde{u}_q=\tilde{u}$ in $C^{2k}_{loc}(\rn)$ where
\begin{equation}\label{eq:tu}
0\leq \tilde{u}\leq \tilde{u}(0)=1\hbox{ and }\Delta_\xi^k\tilde{u}=\mu_{\crit}f(x_\infty)\tilde{u}^{\crit-1}\hbox{ in }\rn,
\end{equation}
and $x_\infty:=\lim_{q\to\crit}x_q$.

\smallskip\noindent We prove the claim. It follows of the naturality of the geometric operator $P_g$ and of \eqref{eq:uq} that
\begin{equation}\label{eq:tuq}
P_{g_q} \tilde{u}_q=\mu_q f(\hbox{exp}_{x_q}(\beta_q \cdot)) \tilde{u}_q^q\hbox{ in }B_{\beta_q^{-1}\delta}(0)
\end{equation}
for all $q\in (2,\crit)$, where $g_q:=(\hbox{exp}_{x_q}^\star g)(\beta_q\cdot)$. In particular, since the exponential is a normal chart at $x_q$, we have that $\lim_{q\to \crit}g_q=\xi$ in $C^{2k}_{loc}(\rn)$. Since $0\leq \tilde{u}_q\leq\tilde{u}_q(0)=1$, it follows from standard elliptic theory (see for instance \cite{adn}) that there exists $\tilde{u}\in C^{2k}(\rn)$ such that $\lim_{q\to\crit}\tilde{u}_q=\tilde{u}$ in $C^{2k}_{loc}(\rn)$. In addition, using that $P_\xi =\Delta_\xi^k$, passing to the limit in \eqref{eq:tuq} yields \eqref{eq:tu}. This proves the claim.

\medskip\noindent{\bf Step 3:} We claim that there exists $C>0$ such that
\begin{equation}\label{bnd:a:b}
\alpha_q\leq \beta_q\leq C\alpha_q
\end{equation} 
when $q\to\crit$.

\smallskip\noindent We prove the claim. We fix $R>0$ and we let $q$ be in $(2,\crit)$: a change of variable and Sobolev's embedding yields
$$\int_{B_R(0)}\tilde{u}_q^{\crit}\, dv_{g_q}=\left(\frac{\alpha_q}{\beta_q}\right)^n\int_{B_{R\beta_q}(x_q)}u_q^{\crit}\, dv_g\leq C\left(\frac{\alpha_q}{\beta_q}\right)^n\Vert u_q\Vert_{P_g}^{\crit}$$
for all $q\in (2,\crit)$. Using \eqref{eq:uq} and Proposition \ref{prop:lim:muq}, letting $q\to \crit$, we get that
$$\left(\frac{\beta_q}{\alpha_q}\right)^n\leq \frac{C'}{\int_{B_R(0)}\tilde{u}^{\crit}\, dv_{\xi}}+o(1)$$
when $q\to\crit$. Since $\tilde{u}(0)>0$, we the get that $\beta_q=O(\alpha_q)$ when $q\to\crit$. This inequality combined with \eqref{lim:alpha} yields \eqref{bnd:a:b}. This proves the claim.

\medskip\noindent{\bf Step 4:} We claim that $\tilde{u}\in D_k^2(\rn)$.

\smallskip\noindent We prove the claim. Indeed, for all $i\in \{0,...,k\}$, it follows from \eqref{bnd:a:b} and a change of variable that $\Vert \nabla^i\tilde{u}_q\Vert_{L^{p_i}(B_R(0))}\leq C\Vert \nabla^i u_q\Vert_{L^{p_i}(B_{R\beta_q}(x_q))}\leq \Vert \nabla^i\tilde{u}_q\Vert_{L^{p_i}(M)}$ for all $q\in (2,\crit)$, all $R>0$ and where $p_i:=\frac{2n}{n-2k+2i}$. It follows from Sobolev's inequalities that the right-hand-side is dominated by $\Vert u_q\Vert_{H_k^2}$, and therefore, letting $q\to\crit$ and $R\to +\infty$ yields $\nabla^i\tilde{u}\in L^{p_i}(\rn)$ for all $i\in \{0,...,k\}$. We let $\eta\in C^\infty_c(\rn)$ be such that $\eta_{|B_1(0)}\equiv 1$: as easily checked, $(\eta(m^{-1}\cdot)\tilde{u})_m\in C^\infty_c(\rn)$ is a Cauchy sequence for the $D_k^2-$norm, and therefore $\tilde{u}\in D_k^2(\rn)$. This proves the claim.

\medskip\noindent{\bf Step 5:} We claim that 
\begin{equation}\label{eq:mu}
\mu_{\crit}=\frac{|O_G(x_\infty)|^{\frac{2k}{n}}}{f(x_\infty)^{\frac{2}{\crit}}K(n,k)}\hbox{ and }\lim_{\alpha\to +\infty}\frac{\beta_q}{\alpha_q}=1
\end{equation}
\smallskip\noindent We prove the claim. Since $\tilde{u}\in D_k^2(\rn)$, we multiply \eqref{eq:tu} by $\tilde{u}$ and integrate to get $\int_{\rn}(\Delta_\xi^{\frac{k}{2}}\tilde{u})^2\, dv_\xi=\mu_{\crit}f(x_\infty)\int_{\rn}\tilde{u}^{\crit}\, dv_\xi$. Since $\tilde{u}\not\equiv 0$, plugging this identity in the Sobolev inequality \eqref{def:K} yields
\begin{equation}\label{lower:bnd:tu}
\int_{\rn}\tilde{u}^{\crit}\, dv_\xi\geq \left(\frac{1}{\mu_{\crit} f(x_\infty)K(n,k)}\right)^{\frac{\crit}{\crit-2}}
\end{equation}
We let $m:=|O_G(x_\infty)|$ if $|O_G(x_\infty)|<\infty$, and any $m\in\nn\setminus\{0\}$ otherwise. We let $\sigma_1,...,\sigma_m\in G$ be such that $\sigma_i(x_\infty)\neq \sigma_j(x_\infty)$ for all $i,j\in \{1,...,m\}$, $i\neq j$. We fix $\delta<\min_{i\neq j}\{d_g(z,z')/\, z\neq z'\in O_G(x_\infty)\}$. The $G-$invariance yields
\begin{eqnarray}
1&=& \int_M f u_q^q\, dv_g\geq \sum_{i=1}^m \int_{B_\delta(\sigma_i(x_\infty))}f u_q^q\, dv_g=m\int_{B_\delta(x_\infty)}f u_q^q\, dv_g\label{eq:mass}\\
&\geq& m\int_{B_{R\beta_q}(x_q)}f u_q^q\, dv_g=m\left(\frac{\beta_q}{\alpha_q}\right)^{n-2k}\int_{B_R(0)}f(\hbox{exp}_{x_q}(\beta_q\cdot)) \tilde{u}_q^q\, dv_{g_q}\nonumber
\end{eqnarray}
for all $q\in (2, \crit)$ and all $R>0$. Letting $q\to +\infty$, and then $R\to +\infty$ and using \eqref{lower:bnd:tu}, we get that
$$1\geq \left(\lim_{q\to\crit}\frac{\beta_q}{\alpha_q}\right)^{n-2k}\frac{m f(x_\infty)}{\left(\mu_{\crit} f(x_\infty)K(n,k)\right)^{\frac{\crit}{\crit-2}}}.$$
In particular, since $\beta_q\geq\alpha_q$ with \eqref{bnd:a:b}, we get an upper-bound for $m$, and therefore $|O_G(x)|<\infty$, and we take $m=|O_G(x)|$. The inequality rewrites
$$\mu_f({\mathcal C}_G) \geq \frac{2}{n-2k}\cdot\frac{|O_G(x_\infty)|^\frac{2k}{n}}{f(x_\infty)^\frac{2}{\crit}K(n,k)}\cdot\left(\lim_{q\to\crit}\frac{\beta_q}{\alpha_q}\right)^{\frac{2k(n-2k)}{n}}.$$
It then follows from \eqref{ineq:large} and \eqref{bnd:a:b} that \eqref{eq:mu} holds. Moreover, we also get that equality holds in \eqref{lower:bnd:tu} and that $\tilde{u}$ is an extremal for the Sobolev inequality \eqref{def:K}. This proves the claim.

\medskip\noindent{\bf Step 6:} We claim that 
\begin{equation}\label{lim:weak}
fu_q^q\, dv_g\rightharpoonup \frac{1}{|O_G(x)|}\delta_{O_G(x)}\hbox{ in the sense of measure when }q\to\crit.
\end{equation}

\smallskip\noindent We prove the claim. Since equality holds in \eqref{lower:bnd:tu}, that $\lim_{q\to\crit}\frac{\alpha_q}{\beta_q}=1$ and that \eqref{eq:mu} holds, we get with a change of variables that
\begin{equation}\label{lim:orbit}
\lim_{R\to +\infty}\lim_{q\to\crit}\int_{B_{R\beta_q}(x_q)}f u_q^q\, dv_g=f(x_\infty)\int_{\rn}\tilde{u}^{\crit}\, dv_\xi=\frac{1}{m}.
\end{equation}
For $\delta>0$, we let $B_\delta(O_G(x_\infty))$ be the union of balls of radius $\delta$ centered at the orbit. Therefore, since $\int_M f u_q^q\, dv_g=1$, \eqref{eq:mass}, \eqref{lim:orbit}  and  the $G-$invariance yield
\begin{equation}\label{norm:out:orbit}
\lim_{q\to\crit}\int_{M\setminus B_\delta(O_G(x_\infty))} f u_q^q\, dv_g=0
\end{equation}
for all $\delta>0$. Consequently, $\lim_{q\to\crit}\int_{B_\delta(z)}f u_q^q\, dv_g=\frac{1}{m}$ for all $\delta>0$ small enough and all $z\in O_G(x)$. Assertion \eqref{lim:weak} then follows. This proves the claim.

\medskip\noindent{\bf Step 7:} We claim that there exists $C>0$ such that
\begin{equation}\label{est:weak}
d(x, O_G(x_q))^{\frac{n-2k}{2}}u_q(x)\leq C
\end{equation}
for all $x\in M$ and all $q\in (2, \crit)$. 

\smallskip\noindent We prove the claim. This pointwise inequality has its origins in Druet \cite{druet:ma}. We define $w_q(x):=d(x, O_G(x_q))^{\frac{n-2k}{2}}u_q(x)$ for all $q\in (2,\crit)$ and all $x\in M$. We argue by contradiction and assume that $\lim_{q\to \crit}\Vert w_q\Vert_\infty=+\infty$. We define $(y_q)_{q\in (2,\crit)}\in M$ such that
\begin{equation}\label{lim:wq}
\max_{y\in M}w_q(y)=w_q(y_q)\to +\infty
\end{equation}
when $q\to \crit$. We define $\gamma_q:=u_q(y_q)^{-\frac{2}{n-2k}}$ for all $q\in (2,\crit)$. It follows from \eqref{lim:wq} that
\begin{equation}\label{lim:gamma}
\lim_{q\to\crit}u_q(y_q)=+\infty\hbox{ and }\lim_{q\to\crit}\gamma_q=0.
\end{equation}
As easily checked, coming back to the definitions of $\alpha_q$ and $\beta_q$, it follows from \eqref{eq:mu} that $\lim_{q\to\crit}u_q(x_q)^{\crit-q}=1$. Therefore, since $u_q(y_q)\leq u_q(x_q)$ for all $q$ and \eqref{lim:gamma} holds, we get that $\lim_{q\to\crit}\gamma_q^{\crit-q}=1$. We define
$$\bar{u}_q(x):=\gamma_q^{\frac{n-2k}{2}}u_q(\hbox{exp}_{y_q}(\gamma_q x))$$
for all $q\in (2,\crit)$ and all $x\in B_{\delta\gamma_q^{-1}}(0)$ where $\delta\in (0, i_g(M))$. Arguing as in Step 2 and using that $\lim_{q\to\crit}\gamma_q^{\crit-q}=1$, we get that
\begin{equation}\label{eq:bu}
P_{\bar{g}_q}\bar{u}_q=\mu_q (1+o(1))f(\hbox{exp}_{y_q}(\gamma_q\cdot))\bar{u}_q^q\hbox{ in }B_{\delta\gamma_q^{-1}}(0)
\end{equation}
for all $q\in (2,\crit)$, where $\lim_{q\to\crit}o(1)=0$ uniformly. We fix $R>0$. It follows from the definition \eqref{lim:wq} of $w_q$ and $y_q$ that
\begin{equation}\label{low:1}
d(\hbox{exp}_{y_q}(\gamma_q x), O_G(x_q))^{\frac{n-2k}{2}}\bar{u}_q(x)\leq d(y_q, O_G(x_q))^{\frac{n-2k}{2}}
\end{equation}
for all $x\in B_R(0)$ and $q\in (2,\crit)$. The limit $w_q(y_q)\to +\infty$ when $q\to \crit$ rewrites $\lim_{q\to\crit}\gamma_q^{-1}d_g(y_q, O_G(x_q))=+\infty$: therefore, there exists $q_0\in (2,\crit)$ such that $d(\hbox{exp}_{y_q}(\gamma_q x), O_G(x_q))\geq d(y_q, O_G(x_q))/2$ for all $x\in B_R(0)$ and all $q\in (q_0,\crit)$, and it follows from \eqref{low:1} that $0\leq \bar{u}_q(x)\leq 2^{\frac{n-2k}{2}}$ for all $x\in B_R(0)$ and all $q\in (q_0, \crit)$. It then follows from \eqref{eq:bu} and standard elliptic theory (see for instance \cite{adn}) that there exists $\bar{u}\in C^{2k}(\rn)$ such that $\lim_{q\to\crit}\bar{u}_q=\bar{u}$ in $C^{2k}_{loc}(\rn)$. Moreover, $\bar{u}\geq 0$ and $\bar{u}(0)=\lim_{q\to\crit}\bar{u}_q(0)=1$, and then $\bar{u}\not\equiv 0$. In particular,
\begin{equation}\label{lim:gamma:y}
\lim_{R\to +\infty}\lim_{q\to\crit}\int_{B_{R\gamma_q}(y_q)}f u_q^q\, dv_g=f(y_\infty)\int_{\rn}\bar{u}^{\crit}\, dv_\xi
\end{equation}
where $y_\infty:=\lim_{q\to\crit}y_q$. Since $\lim_{q\to\crit}\gamma_q^{-1}d_g(y_q, O_G(x_q))=+\infty$ and $\gamma_q\geq \alpha_q=(1+o(1))\beta_q$ when $q\to\crit$, we get that for any $R,R'>0$
$$B_{R\gamma_q}(y_q)\cap B_{R'\beta_q}(O_G(x_q))=\emptyset$$
where $q\to\crit$. We let $\sigma_1,...,\sigma_m\in G$ be such that $O_G(x_\infty)=\{\sigma_1(x_\infty),...,\sigma_m(x_\infty)\}$ and these points are distinct: as easily checked, we have that $\cup_{i=1}^m B_{R'\beta_q}(\sigma_i(x_q))\subset B_{R'\beta_q}(O_G(x_q))$ and the balls are distinct. Therefore
$$1=\int_M f u_q^q\, dv_g\geq \int_{B_{R\gamma_q}(y_q)}f u_q^q\, dv_g+\sum_{i=1}^m \int_{B_{R'\beta_q}(\sigma_i(x_q))}f u_q^q\, dv_g$$
for all $q\in (2,\crit)$ and $R, R'>0$. Letting $q\to\crit$, then $R,R'\to +\infty$ and using \eqref{lim:orbit} and \eqref{lim:gamma:y}, we get that
$$1\geq f(y_\infty)\int_{\rn}\bar{u}^{\crit}\, dv_\xi+1,$$
a contradiction since $\bar{u}\not\equiv 0$. Then \eqref{lim:wq} does not hold and therefore \eqref{est:weak} holds. This proves the claim. 

\medskip\noindent{\bf Step 8:} We claim that 
\begin{equation}\label{cv:out:orbit}
\lim_{q\to\crit}u_q=0\hbox{ in }C^{2k}_{loc}(M\setminus O_G(x_\infty)).
\end{equation}
\smallskip\noindent We prove the claim. We fix $\Omega\subset\subset M\setminus O_G(x_\infty)$ a relatively compact subset. It follows from Step 7 that there exists $C(\Omega)>0$ such that $u_q(x)\leq C(\Omega)$ for all $x\in \Omega$ and all $q\in (2,\crit)$. It then follows from \eqref{eq:uq} and standard elliptic theory (see for instance \cite{adn}) that there exists $u_\infty\in C^\infty(M\setminus O_G(x_\infty))$ such that $\lim_{q\to \crit}u_q=u_\infty$ in $C^{2k}_{loc}(\Omega)$. It then follows from \eqref{norm:out:orbit} that $u_\infty\equiv 0$, and then \eqref{cv:out:orbit} holds. This proves the claim.

\medskip\noindent The following remark will be useful in the sequel: since $\Vert u_q\Vert_{P_g}^2=\mu_q\to\mu_{\crit}$ when $q\to\crit$ and $u_q\to 0$ in $C^{2k}$ outside the orbit, we get from the compact embedding $H_k^2\hookrightarrow H_{k-1}^2$ that 
\begin{equation}\label{lim:strong}
\lim_{q\to\crit}u_q=0\hbox{ strongly in }H_{k-1}^2(M)
\end{equation}

\medskip\noindent{\bf Step 9:} We claim that $\nabla f(x_\infty)=0$.

\smallskip\noindent We prove the claim. Indeed, this is equivalent to proving that $X(f)(x_\infty)=0$ for all vector field $X$ on $M$. With no loss of generality, we assume that $\nabla X (x_\infty)=0$ (this is always possible by modifying $X$ in a normal chart at $x_\infty$) and that $X$ has its support in $B_\delta(x_\infty)$, where $\delta<\min\{d_g(z, z')/\, z\neq z'\in O_G(x_\infty)\}$. We are going to estimate $\int_M X(u_q)\Delta_g^k u_q\, dv_g$ with two different methods. We detail here the case $k=2l$ even and we leave the odd case to the reader.

\smallskip\noindent Integrating by parts, we have that
\begin{eqnarray*}
&&\int_M X(u_q)\Delta^{2l}_g u_q\, dv_g= \int_M \Delta_g^l(X(u_q))\Delta_g^l u_q\, dv_g= \int_M X(\Delta^l_g u_q)\Delta_g^l u_q\, dv_g\\
&&+\sum_{i=1}^{l}\int_M \Delta_g^l u_q\Delta_g^{l-i}\left(\Delta_g(X(\Delta_g^{i-1}u_q))-X(\Delta_g^{i}u_q)\right)\, dv_g. 
\end{eqnarray*}
Using the explicit contraction in \eqref{action:sigma}, we get that 
$$\Delta_g(X(v))-X(\Delta_g v)=(\Delta_g X)(\nabla v)-2(\nabla X,\nabla^2 v)-Ric_g(X, \nabla v),$$
where $v\in C^\infty(M)$ and $\Delta_g X$ is the rough Laplacian, that is $(\Delta_g X)^\alpha=-g^{ij}\nabla_{ij}X^\alpha$. Therefore, we have that (for convenience, we omit the curvature tensor $R$)
$$\Delta_g(X(\Delta_g^{i-1}u_q))-X(\Delta_g^{i}u_q)=\nabla^2 X\star \nabla^{2i-1}u_q+\nabla X\star \nabla^{2i} u_q+X\star \nabla^{2i-1}u_q$$
for all $i\in \{1,...,l\}$, and then,  denoting as $\nabla^{\{m\}}T$ any linear combination of covariant derivatives of $T$ up to order $m$, we get that
\begin{eqnarray*}
&&\Delta_g^{l-i}\left(\Delta_g(X(\Delta_g^{i-1}u_q))-X(\Delta_g^{i}u_q)\right)\\
&&=\Delta_g^{l-i}(\nabla^2 X\star \nabla^{2i-1}u_q+\nabla X\star \nabla^{2i} u_q+X\star  \nabla^{2i-1}u_q)\\
&&= \nabla^{\{2l-2i+2\}}X\star\nabla^{\{2l-1\}}u_q +\nabla X\star\nabla^{2l}u_q,
\end{eqnarray*}
and then
\begin{eqnarray*}
&&\int_M X(u_q)\Delta^{2l}_g u_q\, dv_g=  \int_M X(\Delta^l_g u_q)\Delta_g^l u_q\, dv_g\\
&&+\int_M \Delta_g^l u_q\left(\nabla^{\{2+2l\}}X\star\nabla^{\{2l-1\}}u_q +\nabla X\star\nabla^{2l}u_q\right)\, dv_g\\
&=& \int_M X(\Delta^l_g u_q)\Delta_g^l u_q\, dv_g+\int_M \Delta_g^l u_q \star \nabla^{\{2l+2\}}X\star \nabla^{\{2l-1\}}u_q\, dv_g \\
&&+\int_{M}\Delta_g^{l}u_q \nabla X\star \nabla^{2l}u_q\, dv_g
\end{eqnarray*}
Since $k=2l$, it follows from \eqref{lim:strong} and the Cauchy-Schwarz inequality that
$$\int_M \Delta_g^l u_q  \star \nabla^{\{2l+2\}}X\star \nabla^{\{2l-1\}}u_q\, dv_g =O\left(\Vert u_q\Vert_{H_k^2}\Vert u_q\Vert_{H_{k-1}^2}\right)=o(1)$$
when $q\to\crit$. Moreover, since $\nabla X(x_\infty)=0$ and \eqref{cv:out:orbit} holds, we get that 
$$\int_{M}\Delta_g^{l}u_q \nabla X\star \nabla^{2l}u_q\, dv_g=o(\Vert u_q\Vert_{H_k^2})=o(1) $$
when $q\to\crit$. Therefore, integrating by parts, we get that
\begin{eqnarray*}
&&\int_M X(u_q)\Delta^{2l}_g u_q\, dv_g= \int_M X(\Delta^l_g u_q)\Delta_g^l u_q\, dv_g+o(1)\\
&&= \int_M X\left(\frac{(\Delta_g^l u_q)^2}{2}\right)\, dv_g+o(1)= -\int_M \frac{\hbox{div}_g(X)}{2}(\Delta^l u_q)^2+o(1)
\end{eqnarray*}
when $q\to \crit$ and where $\hbox{div}_g(X)=\nabla_iX^i$. Since $\nabla X(x_\infty)=0$, \eqref{cv:out:orbit} holds and $\Vert u_q\Vert_{H_k^2}\leq C$ for all $q\to\crit$, we get that the right-hand-side above goes to zero, and then
\begin{equation}\label{lim:1}
\lim_{q\to\crit}\int_M X(u_q)\Delta^{2l}_g u_q\, dv_g=0.
\end{equation}

\medskip\noindent We now estimate $\int_M X(u_q)\Delta^{2l}_gu_q\, dv_g$ using equation \eqref{eq:uq}. It follows from \eqref{eq:Pg:adj} that
$$\int_M X(u_q) P_g u_q\, dv_g=\int_M\Delta_g^{l}X(u_q)\Delta^{l}_g u_q\, dv_g+ \sum_{l=0}^{k-1}\int_M A_{(l)}(\nabla^l X(u_q),\nabla^l u_q)\, dv_g$$
It then follows from \eqref{lim:strong} and an integration by parts that
\begin{equation*}
\int_M X(u_q)\Delta^{2l}_g u_q\, dv_g=\int_M X(u_q) P_g u_q\, dv_g+o(1)
\end{equation*}
when $q\to\crit$. We now use equation \eqref{eq:uq} to get that
\begin{eqnarray*}
&&\int_M X(u_q)\Delta^{2l}_g u_q\, dv_g=\mu_q \int_M f X(u_q)u_q^{q-1}\, dv_g+o(1)\\
&&=\mu_q\int_M f X\left(\frac{u_q^q}{q}\right)\, dv_g= -\frac{\mu_q}{q}\int_M (X(f)+f\hbox{div}_g(X))u_q^q\, dv_g+o(1)
\end{eqnarray*}
when $q\to\crit$. It now follows from Proposition \ref{prop:lim:muq}, \eqref{lim:weak} and $\nabla X(x_\infty)=0$ that
$$\lim_{q\to\crit}\int_M X(u_q)\Delta^{2l}_g u_q\, dv_g=-\frac{\mu_{\crit} X(f)(x_\infty)}{\crit |O_G(x_\infty)|f(x_\infty)}.$$
This limit combined with \eqref{lim:1} yields $X(f)(x_\infty)=0$, which, as already mentioned, proves that $\nabla f(x_\infty)=0$. This ends Step 9.

\medskip\noindent Theorem \ref{th:quanti} is a direct consequence of Steps 1 to 9.

\medskip\noindent As a direct byproduct of Theorem \ref{th:quanti}, we have the following proposition:
\begin{prop}\label{prop:strict} Let $(M, {\mathcal C})$ be a conformal Riemannian manifold of dimension $n\geq 3$ and let $k\in \nn^\star$ be such that $2k<n$. Let $G$ be a group of diffeomorphisms such that ${\mathcal C}_G\neq\emptyset$ and let $f\in C^\infty_{G,+}(M)$ be a positive $G-$invariant function. Assume that there exists $g\in {\mathcal C}_G$ such that $P_g$ satisfies $(C)$ and $(PPP)$. We assume that
$$\mu_f({\mathcal C}_G)<\frac{2}{n-2k}\cdot\frac{|O_G(x)|^{\frac{2k}{n}}}{f(x)^{\frac{2}{\crit}}K(n,k)},$$
for all $x\in M$. Then there exists $\hat{g}\in {\mathcal C}_G$ such that $Q_{\hat{g}}=f$ and the infimum $\mu_f({\mathcal C}_G)$ is achieved.
\end{prop}
A similar result was proved in \cite{hebeysphere} for $k=1$ and in \cite{bfr} when $n=2k$.

\section{The case of the sphere}
We consider here the standard unit $n-$sphere $\sn$ endowed with its standard round metric $h$ and the associated conformal class ${\mathcal C}:=[h]$.

\begin{prop}\label{prop:test:sphere} Let $G$ be a subgroup of $Isom_h(\sn)$ and let $f\in C^\infty_{G,+}(\sn)$ be a smooth positive function. Let $p\in\sn$ be such that $\nabla^if(p)=0$ for all $i\in\{1,...,n-2k\}$ and $|O_G(p)|\geq 2$. Then
\begin{equation*}
\mu_f({\mathcal C}_G)<\frac{2}{n-2k}\cdot\frac{|O_G(p)|^\frac{2k}{n}}{K(n,k)f(p)^{\frac{2}{\crit}}}.
\end{equation*}
\end{prop}
\begin{proof} Given $\lambda>1$ and $x_0\in \sn$, we let $\phi_\lambda: \sn\to\sn$ be such that $\phi_\lambda(x)=\pi_{x_0}^{-1}(\lambda^{-1}\pi_{x_0}(x) )$ if $x\neq x_0$ and $\phi_\lambda(x_0)=x_0$, where $\pi_{x_0}$ is the stereographic projection of pole $x_0$. Up to a rotation, we can assume that $x_0:=(0,...,0,1)$ is the north pole: then we have that 
$(\pi_N^{-1})^\star h=U_1^{\frac{4}{n-2k}}\xi$, where $U_1(x):=\left((1+|x|^2)/2\right)^{k-n/2}$. As easily checked, $\phi_\lambda$ is a conformal diffeomorphism and standard computations yield $\phi_\lambda^\star h=u_{x_0,\beta}^{\frac{4}{n-2k}}h$ where $\beta:=(\lambda^2+1)(\lambda^2-1)^{-1}$ and 
$$u_{x_0,\beta}(x):=\left(\frac{\sqrt{\beta^2-1}}{\beta-\cos d_h(x,x_0)}\right)^{\frac{n-2k}{2}}$$
for all $x\in\sn$ and $\beta>1$. In particular, we have that
\begin{equation}\label{eq:vol}
\int_{\sn}u_{p,\beta}^{\crit}\, dv_h=\omega_n
\end{equation}
where $\omega_{n}>0$ is the volume of $(\sn,h)$. It follows from the conformal law \eqref{eq:conf} that
\begin{equation}\label{eq:u:beta}
P_h u_{x_0,\beta}=c_{n,k}Q_h u_{x_0,\beta}^{\crit-1}\hbox{ in }\sn\hbox{ with }c_{n,k}:=\frac{n-2k}{2}. 
\end{equation}
We now fix $p\in\sn$ as in the statement of Proposition \ref{prop:test:sphere} and we let $\sigma_1,...\sigma_m\in G$ be such that $O_G(p)=\{\sigma_1(p),...,\sigma_m(p)\}$ and $|O_G(p)|=m\geq 2$. We define
$$u_\beta:=\sum_{i=1}^m u_{\sigma_i(p),\beta}$$
for all $\beta>1$. One checks that $u_\beta$ is positive and $G-$invariant. Let us estimate
$$I_h(u_\beta):=\frac{\int_{\sn} u_\beta P_h u_\beta\, dv_h}{\left(\int_{\sn} f u_\beta ^{\crit}\, dv_h\right)^{\frac{2}{\crit}}}.$$
The $G-$invariance and \eqref{eq:u:beta} yield
\begin{eqnarray*}
\int_{\sn} u_\beta P_h u_\beta\, dv_h& = & c_{n,k}Q_h \sum_{i,j=1}^m\int_{\sn}u_{\sigma_i(p),\beta}u_{\sigma_j(p),\beta}^{\crit-1}\, dv_h= m c_{n,k}Q_h \left(\omega_{n}+d_\beta\right)
\end{eqnarray*}
where we have used \eqref{eq:vol} and where $$d_\beta:=\sum_{i=2}^m\int_{\sn}u_{\beta,p}u_{\beta,\sigma_i(p)}^{\crit-1}\, dv_h$$
for all $\beta>1$. Standard computations yield
$$d_\beta=(1+o(1))\Lambda_{p, G}(\beta^2-1)^{\frac{n-2k}{2}}$$
when $\beta\to 1$, where
$$\Lambda_{p, G}:=\left(\int_{\sn}(1-\cos d_h(x,p))^{k-n/2}\, dv_h\right)\cdot \sum_{i=2}^m\left(1-\cos d_h (p,\sigma_i(p))\right)^{k-n/2}\, dv_h>0.$$ 
Concerning the denominator, it follows from the cancelation hypothesis on the derivatives of $f$ that $|f(x)-f(p)|\leq C d_h(x,O_G(p))^{n-2k+1}$ for all $x\in \sn$. Therefore, rough estimates yield
\begin{equation*}
\left|\int_{\sn}(f-f(p)) u_{\beta}^{\crit}\, dv_h\right|\leq C(\beta^2-1)^{\frac{n-2k+1}{2}}
\end{equation*}
for all $\beta>1$. A convexity inequality yields
\begin{eqnarray*}
\int_{\sn}u_\beta^{\crit}\, dv_h&\geq &\sum_{i=1}^m\int_{\sn}u_{\beta,\sigma_i(p)}^{\crit}\, dv_h+\crit\sum_{i\neq j}\int_{\sn}u_{\sigma_i(p),\beta}u_{\sigma_j(p),\beta}^{\crit-1}\, dv_h\\
&\geq & m\left(\omega_{n}+\crit d_\beta\right)
\end{eqnarray*}
Noting $\Lambda_{p,G}>0$ and that $c_{n,k}Q_h\omega_n^{\frac{\crit-2}{\crit}}=K(n,k)^{-1}$ (since pulling back $u_{\beta,p}$ by the stereographic projections gives $U_1$, an extremal for \eqref{def:K}), these estimates yield
\begin{eqnarray*}
I_h(u_\beta) &\leq &  \frac{|O_G(p)|^\frac{2k}{n}}{f(p)^{\frac{2}{\crit}}K(n,k)}\cdot\left(1-\frac{\Lambda_{p, G}}{\omega_n}(\beta^2-1)^{\frac{n-2k}{2}}+o((\beta^2-1)^{\frac{n-2k}{2}})\right)\\
&<& \frac{|O_G(p)|^\frac{2k}{n}}{f(p)^{\frac{2}{\crit}}K(n,k)}.
\end{eqnarray*}
Coming back to the definition of $\mu_f({\mathcal C}_G)$, this proves Proposition \ref{prop:test:sphere}.\end{proof}

\medskip\noindent{\bf Proof of Theorem \ref{th:main}:} In the case $n=2k+1$, it follows from Proposition \ref{prop:ppp} and \ref{prop:test:sphere} that Case (i) of Theorem \ref{th:quanti} cannot hold. Therefore Case (ii) holds, and Theorem \ref{th:main} is proved.\par
\smallskip\noindent More generally, Propositions \ref{prop:ppp} and \ref{prop:large}, Theorem \ref{th:quanti} and Proposition \ref{prop:test:sphere}, yield:
\begin{thm}\label{thm:n} Let $k\geq 1$ and let $G$ be a subgroup of isometries of $(\mathbb{S}^{n},h)$, $n>2k$. Let $f\in C^\infty(M)$ be a positive $G-$invariant function and assume that $G$ acts without fixed point (that is $|O_G(x)|\geq 2$ for all $x\in \mathbb{S}^{n}$).  Assume that there exists $p\in \sn$ such that 
$$\frac{|O_G(p)|^\frac{2k}{n}}{f(p)^{\frac{2}{\crit}}}\leq \frac{|O_G(x)|^\frac{2k}{n}}{f(x)^{\frac{2}{\crit}}}$$
for all $x\in \sn$ and that $\nabla^if(p)=0$ for all $i\in\{1,...,n-2k\}$. Then there exists $g\in [h]$ such that $Q_g=f$ and $G\subset Isom_g(\sn)$.
\end{thm}

\end{document}